\documentclass[10pt,reqno]{amsart}
\usepackage[a4paper,%
            includefoot,  
            marginparwidth=0in,     
            marginparsep=0in,       
            margin=1.45in,               
            includemp]{geometry}
\PassOptionsToPackage{hyphens}{url}
\usepackage[hidelinks,unicode=true]{hyperref}
\usepackage{fourier,amssymb,amsfonts,mathrsfs}
\usepackage{amsmath,amsthm}
\usepackage{enumitem}
\usepackage[authoryear]{natbib}
\newtheorem{theorem}{Theorem}[section]
\newtheorem{lemma}[theorem]{Lemma}
\newtheorem{definition}[theorem]{Definition}
\newtheorem{remark}[theorem]{Remark}
\newtheorem{corollary}[theorem]{Corollary}
\newtheorem{question}[theorem]{Question}
\numberwithin{equation}{section}
\newcommand{\DTV}{\mathrm{d}_{\mathrm{TV}}}
\title[On cutoff via rigidity for high-dimensional curved diffusions]{On cutoff via rigidity for\\high-dimensional curved diffusions}
\date{Autumn 2024; Revised Spring 2025; Published Autumn 2025; Revised Winter
  2026; Autumn 2026. Postpublication version for \emph{Comptes Rendus.
    Mathématique 2025, Vol.~363, p.~1103-1121} DOI:
  \href{https://doi.org/10.5802/crmath.776}{10.5802/crmath.776}}
\author{Djalil Chafaï}
{\footnotesize\address[DC]{DMA, École normale supérieure, 45 rue d'Ulm, 75230, Paris; and CEREMADE, Université Paris-Dauphine, Place du Maréchal de Lattre de Tassigny, 75775, Paris; PSL, CNRS}}
\email{djalil@chafai.net}
\urladdr{https://djalil.chafai.net/}
\author{Max Fathi}
{\footnotesize\address[MF]{Université Paris Cité and Sorbonne Université, CNRS, Laboratoire Jacques-Louis Lions and Laboratoire de Probabilités, Statistique et Modélisation, F-75013 Paris, France\\
and DMA, École normale supérieure, Université PSL, CNRS, 75005 Paris, France \\
and Institut Universitaire de France}}
\email{mfathi@lpsm.paris}
\begin{document}

\begin{abstract}
  We consider overdamped Langevin diffusions in Euclidean space, with
  curvature equal to the spectral gap. This includes the
  Ornstein--Uhlenbeck process as well as non-Gaussian and non-product
  extensions with convex interaction, such as the Dyson process from random
  matrix theory. We show that a cutoff phenomenon or abrupt convergence to
  equilibrium occurs in high dimension, at a critical time equal to the
  logarithm of the dimension divided by twice the spectral gap. This cutoff
  holds for Wasserstein distance, total variation, relative entropy, and
  Fisher information. A key observation is a relation to a spectral rigidity,
  linked to the presence of a Gaussian factor. A novelty is the extensive
  usage of functional inequalities, even for short-time regularization, and
  the reduction to Wasserstein. The proofs are short and conceptual. Since the
  product condition is satisfied, an Lp cutoff holds for all finite p. We moreover
  discuss a natural extension to Riemannian manifolds, a link with logarithmic
  gradient estimates in short-time for the heat kernel, and ask about
  stability by perturbation. Finally, beyond rigidity but still for
  diffusions, a discussion around the recent progress on the product condition
  for non-negatively curved diffusions leads us to introduce a new curvature
  product condition.
\end{abstract}

\maketitle

{\footnotesize\tableofcontents}

\section{Introduction and main results}

\subsection{Diffusions on Euclidean spaces, with convex potential}

Let ${(X_t)}_{t\geq0}$ be the Markov diffusion process solving the stochastic differential equation (SDE)
\begin{equation} \label{eq:eds}
\mathrm{d}X_t = -\nabla V(X_t)\mathrm{d}t + \sqrt{2}\mathrm{d}B_t,
\quad X_0=x_0\in\mathbb{R}^d,
\end{equation}
where ${(B_t)}_{t\geq0}$ is a standard Brownian motion in $\mathbb{R}^d$, $V:\mathbb{R}^d\to\mathbb{R}$ is strictly convex and $\mathscr{C}^2$ with
$\lim_{|x|\to\infty}V(x)=+\infty$, and $\left|\cdot\right|$ is the Euclidean norm of $\mathbb{R}^d$. In Statistical Physics, this drift-diffusion is also known as an overdamped Langevin process with potential $V$. By adding a constant to $V$, we can assume without loss of generality that 
$\mu=\mathrm{e}^{-V}$ namely
\begin{equation}\label{eq:mu}
\mathrm{d}\mu(x)=\mathrm{e}^{-V(x)}\mathrm{d}x
\end{equation}
is a probability measure. It is the unique invariant law of the process, and it is moreover reversible. The associated infinitesimal generator is the linear differential operator 
\begin{equation}\label{eq:L}
\mathcal{L} = \Delta - \nabla V \cdot \nabla
\end{equation}
acting on smooth functions. It is symmetric in $L^2(\mu)$, and its kernel is
the set of constant functions. Moreover, its spectrum is included in
$(-\infty,-\lambda_1]\cup\{0\}$, for some $\lambda_1>0$ called the spectral
gap of $\mathcal{L}$. The Ornstein--Uhlenbeck (OU) process is obtained
when $V(x)=\frac{\rho}{2}|x|^2$, $\rho>0$, for which $\lambda_1=\rho$ while
$\mathrm{Hess}(V)(x)=\rho \mathrm{I}_d$ for all $x\in\mathbb{R}^d$.

If $V$ is $\rho$-convex for some $\rho>0$, namely if
$V-\frac{\rho}{2}\left|\cdot\right|^2$ is convex, then the spectral gap is an
eigenvalue of $-\mathcal{L}$, since the spectrum is discrete, see \cite[Prop.~6.7]{GMS15}.

\subsection{Cutoff for high-dimensional curved diffusions}

Let us denote by $\mathrm{W}_2$ the $L^2$ Wasserstein (or
Monge--Kantorovich) coupling distance between probability measures on the
same metric space with finite second moment, namely
\begin{equation}
  \mathrm{W}_2(\nu,\mu)=\inf_{(X,Y)}\sqrt{\mathbb{E}(d(X,Y)^2)}
\end{equation}
where the infimum runs over all couples $(X,Y)$ of random variables with
$X\sim\nu$ and $Y\sim\mu$, see \cite{MR3409718}. We also write
$\mathrm{W}_2(Z,\mu)=\mathrm{W}_2(\mathrm{Law}(Z),\mu)$. We take
$d(x,y)=|x-y|$ on $\mathbb{R}^k$.

\begin{theorem}[Main Wasserstein estimate]\label{th:maineucl}
  Let ${(X_t)}_{t\geq0}$ be the process \eqref{eq:eds}, with potential $V$,
  spectral gap $\lambda_1$, and invariant law $\mu=\mathrm{e}^{-V}$. If, for
  all $x\in\mathbb{R}^d$, and as quadratic forms,
  \begin{equation}\label{eq:curgap}
    \mathrm{Hess}(V)(x) \geq \lambda_1\mathrm{I}_d,
  \end{equation}
  then for all non-empty set of initial conditions $S\subset\mathbb{R}^d$,
  \begin{equation}\label{eq:maineucl}
    \frac{\mathrm{e}^{-2\lambda_1 t}}{\lambda_1}\sup_{x_0 \in
        S}\Lambda(x_0)
       +\frac{k_1\bigl(1-\sqrt{1-\mathrm{e}^{-2\lambda_1 t}}\bigr)^2}{\lambda_1}
    \leq\sup_{x_0\in S}\mathrm{W}_2(X_t, \mu)^2
	  \leq \mathrm{e}^{-2\lambda_1 t} \sup_{x_0 \in S}\int|x-x_0|^2\mathrm{d}\mu(x)
  \end{equation}
  with, denoting $E_1$ the eigenspace of $-\mathcal{L}$ associated to
  $\lambda_1$ and $k_1=\dim(E_1)$,
  \begin{equation}
    \Lambda(x_0)
    =\sup_{(f_1,\ldots,f_{k_1})}\sum_{i=1}^{k_1}|f_i(x_0)|^2,
  \end{equation}
  where the supremum runs over the set of orthonormal bases of $E_1$.\\
  Moreover, if $\mu$ is centered and $S=\{x\in\mathbb{R}^d:|x|\leq R\}$, then
  $\displaystyle\sup_{x_0 \in S}\Lambda(x_0)$ can be replaced by
  $\lambda_1R^2$.
\end{theorem}

\begin{remark}[$k_1$ and $\Lambda(x_0)$]
  We always have $k_1 \leq d$, see \cite[Lem.~14]{BaKl20}. The factor involving $k_1$ could be dropped in the lower bound, since it is nonnegative. This will be good enough for the proof of Corollary \ref{cor_cutoff_langevin_general} below, but there are situations, such as very small sets of initial data, where the extra factor is useful to get sharp bounds.\\
  By elliptic regularity, the elements of $E_1$ are $\mathcal{C}^2$ and in particular continuous, see \cite[Ch.~6]{Evans10}.
  Moreover, since $E_1$ is finite dimensional, the linear map
  $f\in E_1\mapsto\ell(f)=f(x_0)$ is continuous, and by the Riesz
  representation theorem, there exists an element $h_{x_0}\in E_1$ such that
  $\ell=\langle \cdot,h_{x_0}\rangle$, and then the Parseval identity gives
  $\sum_{i=1}^{k_1}|f_i(x_0)|^2=\sum_{i=1}^{k_1}|\langle
  f_i,h_{x_0}\rangle|^2=\|h_{x_0}\|^2$ for all orthonormal basis
  $f_1,\ldots,f_{k_1}$ of $E_1$. Therefore the supremum in the definition of
  $\Lambda(x_0)$ is useless.
\end{remark}

\begin{corollary}[Wasserstein cutoff] \label{cor_cutoff_langevin_general} %
  Let $(X_t^{(d)})$, $V^{(d)}$, $\lambda_1^{(d)}$, $\mu^{(d)}$ be as in
  Theorem~\ref{th:maineucl}, satisfying \eqref{eq:curgap} for any dimension
  $d$. Let $m^{(d)}$ be the mean of $\mu^{(d)}$, and assume that
  \begin{equation}
    \liminf_{d\to\infty}\lambda_1^{(d)} > 0.
  \end{equation}
  Then a cutoff phenomenon occurs at critical time
  \begin{equation}
    t_*=\frac{\log(d)}{2\lambda_1^{(d)}},
  \end{equation}
  namely for all
  $\varepsilon\in(0,1)$,
  \begin{equation}
    \lim_{d\to\infty}
    \sup_{x^{(d)}_0\in S^{(d)}}
    \mathrm{W}_2(X^{(d)}_{t_d},\mu^{(d)})
    =
    \begin{cases}\label{eq:cutoff:W2}
      +\infty & \text{if $t_d=(1-\varepsilon)t_*$}\\
      0 & \text{if $t_d=(1+\varepsilon)t_*$}
    \end{cases}
  \end{equation}
  where the set of initial conditions is a ball of the following form
  \begin{equation}
    S^{(d)}=\{x\in\mathbb{R}^d:|x-m^{(d)}|\leq c\sqrt{d}\},\quad\text{for an arbitrary constant $c>0$}.
  \end{equation}
\end{corollary}

\bigskip

The assumption $c>0$ is essential to get a cutoff at critical time
$\log(d)/(2\lambda_1^{(d)})$. Indeed, in the case of the $d$-dimensional
Ornstein--Uhlenbeck process with $S^{(d)}=\{m^{(d)}\}$ and spectral gap
$\lambda_1^{(d)}=\rho>0$ independent of $d$, one has a cutoff
at critical time $\log(d)/(4\lambda_1^{(d)})$, and this is an effect of the
second part of the left hand side of \eqref{eq:maineucl}.

\bigskip

For probability measures $\mu$ and $\nu$ on the same space, we denote by
\begin{equation}
    \DTV(\mu,\nu)=\sup_A|\mu(A)-\nu(A)|\in[0,1]
\end{equation}
their total variation distance, and by
\begin{equation}
\mathrm{H}(\nu\mid\mu)
=\int\frac{\mathrm{d}\nu}{\mathrm{d}\mu}\log\Bigl(\frac{\mathrm{d}\nu}{\mathrm{d}\mu}\Bigr)\mathrm{d}\mu
=\int\log\Bigl(\frac{\mathrm{d}\nu}{\mathrm{d}\mu}\Bigr)\mathrm{d}\nu
\in[0,+\infty]
\end{equation}
the relative entropy or Kullback--Leibler divergence of $\nu$ with respect
to $\mu$, with convention $\mathrm{H}(\nu\mid\mu)=+\infty$ if $\nu$ is not
absolutely continuous with respect to $\mu$. The Fisher information of $\nu$ with respect to $\mu$ is 
\begin{equation}
  \mathrm{I}(\nu\mid\mu)
  =\int\frac{|\nabla\frac{\mathrm{d}\nu}{\mathrm{d}\mu}|^2}{\frac{\mathrm{d}\nu}{\mathrm{d}\mu}}\mathrm{d}\mu
  =\int\Bigl|\nabla\log\Bigl(\frac{\mathrm{d}\nu}{\mathrm{d}\mu}\Bigr)\Bigr|^2\mathrm{d}\nu
  \in[0,+\infty],
\end{equation}
with the convention $\mathrm{I}(\nu\mid\mu)=+\infty$ if $\nu$ is not absolutely continuous with respect to $\mu$. 
By an abuse of notation, we take the freedom of replacing $\nu$ by $X\sim\nu$ in these expressions. 

\begin{corollary}[I, H, and TV cutoffs] \label{co:cutoff:tv:H} Under the
  setting of Corollary~\ref{cor_cutoff_langevin_general}, and if additionally
  \begin{equation}
    \lambda_1^{(d)}=o_{d\to\infty}(\log(d))
  \end{equation}
  then, for all $\varepsilon\in(0,1)$,
 \begin{align}
     \lim_{d\to\infty}
     \sup_{x^{(d)}_0\in S^{(d)}}
     \mathrm{I}(X^{(d)}_{t_d}\mid\mu^{(d)})
     &=
     \begin{cases}
     +\infty & \text{if $t_d=(1-\varepsilon)t_*$}\\
     0 & \text{if $t_d=(1+\varepsilon)t_*$}
     \end{cases}\label{eq:cutoff:I}\\
     \lim_{d\to\infty}
     \sup_{x^{(d)}_0\in S^{(d)}}
     \mathrm{H}(X^{(d)}_{t_d}\mid\mu^{(d)})
     &=
     \begin{cases}
     +\infty & \text{if $t_d=(1-\varepsilon)t_*$}\\
     0 & \text{if $t_d=(1+\varepsilon)t_*$}
     \end{cases}\label{eq:cutoff:H}\\
     \lim_{d\to\infty}
     \sup_{x^{(d)}_0\in S^{(d)}}
     \DTV(X^{(d)}_{t_d},\mu^{(d)})
     &=
     \begin{cases}
     1 & \text{if $t_d=(1-\varepsilon)t_*$}\\
     0 & \text{if $t_d=(1+\varepsilon)t_*$}
     \end{cases}. \label{eq:cutoff:TV}
 \end{align} 
\end{corollary}

We emphasize that in contrast with what is done for instance by \cite{BCL23}
and \cite{salez-2025}, we obtain the $\mathrm{I}$, $\mathrm{H}$, and
$\mathrm{TV}$ cutoffs from the Wasserstein cutoff. Moreover, they occur at
exactly the same critical time, due to the choice of initial condition we
make. In comparison, \cite{salez-2025} works under more general curvature
assumptions, and uses Fisher information and entropy to derive a differential
inequality controlling entropy beyond the mixing time, and then proves cutoff
in total variation distance using a reverse Pinsker inequality. All these
approaches are crucially based on functional inequalities.

Condition \eqref{eq:curgap} states that the process has curvature at least
equal to the spectral gap. It turns out that it is the best possible lower
bound on the curvature, as we explain later on in relation with a notion of
spectral rigidity. Condition \eqref{eq:curgap} is satisfied by the OU process
with $V=\frac{\rho}{2}\left|\cdot\right|^2$, and in this case, we have
$\lambda_1=\rho$, $k_1=d$, and $f_i(x)=\sqrt{\rho}x_i$. An important class of
non-Gaussian and non-product examples beyond pure OU is
\begin{equation}\label{eq:VW}
    V(x)=\frac{\rho}{2}|x|^2+W(x), \quad x\in\mathbb{R}^d,
\end{equation}
where $\rho>0$ and $W:\mathbb{R}^d\to\mathbb{R}$ is convex and
translation invariant in the direction $(1,\ldots,1)\in\mathbb{R}^d$, namely
for all $u\in\mathbb{R}$ and all $x\in\mathbb{R}^d$,
\begin{equation}
    W(x+u(1,\ldots,1))=W(x).
\end{equation}
This is the case for example when
for some convex even function $h:\mathbb{R}\to\mathbb{R}$,
\begin{equation}
W(x)=\sum_{i<j}h(x_i-x_j),\quad x\in\mathbb{R}^d.
\end{equation}
If $\pi$ and $\pi^\perp$ are the orthogonal projections on
$\mathbb{R}(1,\ldots,1)$ and its orthogonal, respectively, then
$|x|^2=|\pi(x)|^2+|\pi^\perp(x)|^2$, while the translation invariance of $W$
in the direction $(1/\sqrt{d},\ldots,1/\sqrt{d})$ gives
$W(x)=W(\pi(x)+\pi^\perp(x))=W(\pi^\perp(x))$, therefore
\begin{equation}
\mathrm{e}^{-V(x)}
=\mathrm{e}^{-\frac{\rho}{2}|\pi(x)|^2}\mathrm{e}^{-(W(\pi^\perp(x))+\frac{\rho}{2}|\pi^\perp(x)|^2)}
\end{equation}
which means that $\mu$ is, up to a rotation, a product measure, and splits
into a one-dimensional Gaussian factor $\mathcal{N}(0,\frac{1}{\rho})$ and a
log-concave factor with a $\rho$-convex potential.

\begin{theorem}[Boltzmann--Gibbs probability measure with convex
  interaction] \label{thm_bg_convex_inter} Let us consider the Langevin
  process \eqref{eq:eds} with $V$ as in \eqref{eq:VW} with $\rho>0$ and with $W$
  convex and translation invariant in the direction $(1,\ldots,1)$. Then the
  following properties hold true:
  \begin{enumerate}[label=(\roman*)]
  \item $\lambda_1=\rho$ and the symmetric Hermite polynomial $x_1+\cdots+x_d$ belongs to $E_1$. 
  \item $V$ is $\rho$-convex and $\mu$ has a Gaussian factor
    $\mathcal{N}(0,\frac{1}{\rho})$ in the direction $(1,\ldots,1)$.
  \item The curvature condition \eqref{eq:curgap} is satisfied. 
  \end{enumerate}
\end{theorem}

This approach can be adapted to the special degenerate case of the
Dyson--OU or DOU process studied in \cite{MR4175749,BCL23}, when
\begin{equation}
    h(x)=
    \begin{cases}
        -\beta\log(x) & \text{if $x>0$}\\
        +\infty & \text{if $x\leq0$}
    \end{cases},\quad\text{for an arbitrary constant $\beta\geq0$},
\end{equation}
the degeneracy being equivalent to define the Dyson--OU process on
the convex domain $\{x\in\mathbb{R}^d:x_1>\cdots>x_d\}$ instead of on the
whole space $\mathbb{R}^d$, to exploit convexity. In this case, the symmetric
Hermite polynomial $x_1+\cdots+x_d$ is an eigenfunction associated to the
spectral gap, the tip of an iceberg of integrability, as observed
in \cite{MR1133488}.

\begin{corollary}[Cutoff for Langevin with convex interactions]\label{cor_convex_inter} %
  Let $(X_t^{(d)})$, $V^{(d)}$,
  $\mu^{(d)}$ be as in Theorem~\ref{thm_bg_convex_inter}, for any dimension
  $d$, and for a fixed $\rho>0$. Then there is cutoff at critical time
  \begin{equation}
	  t_* =\frac{\log(d)}{2\rho}
  \end{equation}
  in the sense that 
  \eqref{eq:cutoff:W2}, \eqref{eq:cutoff:I}, \eqref{eq:cutoff:H}, and \eqref{eq:cutoff:TV} occur
  when $S^{(d)}$ can be either
  $B(m^{(d)}, c\sqrt{d})$ or $m^{(d)} + [-c, c]^d$ for an arbitrary fixed constant $c>0$ with $m^{(d)}$ being the mean of $\mu^{(d)}$.
\end{corollary}

To summarize, for rigid curved diffusions with $\liminf_{d\to\infty}\lambda_1^{(d)}>0$ and $\lambda_1^{(d)}=o_{d\to\infty}(\log d)$, the mixing time in Wasserstein distance, relative entropy, and total variation distance is always $\frac{\log(d)}{2\lambda_1}$, and a cutoff occurs at the mixing time in high dimension, for an explicit class of initial conditions. This covers non-Gaussian non-product examples, beyond the (D)OU case.

\subsection{Product condition and \texorpdfstring{$L^p$}{Lp} cutoff at the mixing time}

It seems that the cutoff phenomenon for diffusions was first explored by
\cite{Saloff94}, notably for Brownian motion on compact Lie groups, using the
$L^2$ decomposition and functional inequalities. This was further explored for
manifolds with special symmetries in relation with representation theory by
\cite{Meliot14}. The cutoff for diffusions on non-compact spaces such as the
Dyson--OU process, is considered by \cite{BCL23}, for various distances
and divergences, in relation with integrability. The OU process is a special
Dyson--OU process, and is also a special Gaussian ergodic Markov process
with independent components, a tensorized or product situation for which the
cutoff was studied earlier notably by \cite{MR2203823,MR2260742}, and by
\cite{MR3770869}. The role of the eigenfunctions associated to the spectral
gap for lower bounds is an old observation that dates back to Persi Diaconis
and David Wilson, see for instance \cite{Saloff94} and \cite{BCL23} and
references therein for the case of diffusions.

It was shown by \cite{MR2375599} that for
ergodic Markov processes on arbitrary state spaces, cutoff occurs at the
mixing time, in $L^p$ distance, $p>1$, provided that the product of the
spectral gap and the mixing time tends to infinity. This \emph{product
  condition} was proposed by Yuval Peres. The method relies on a reduction to
the Euclidean case $p=2$ by interpolation. It does not provide the mixing time. Very recently,
\cite{salez-2025} extended this theorem to $p=1$ (total variation distance),
in the case of non-negatively curved diffusions. The method relies crucially
on a two-sided comparison between total variation and relative entropy, namely
Pinsker and reverse Pinsker inequalities, in other words functional
inequalities. It is a variant of the method developed by \cite{Salez24} for
finite state spaces, with a crucial specificity. Namely, it is shown that in
the continuous setting a varentropy functional is controlled by the derivative
of the entropy via a local Poincar\'e inequality, which arises as a
consequence of non-negative curvature. In the discrete setting, chain rule
issues are an obstruction, and current known statements have some extra
restrictions, see \cite{pedrotti-2025}.

Under the hypotheses of Corollary~\ref{co:cutoff:tv:H}, including
$\liminf_{d\to\infty}\lambda_1^{(d)}>0$ and
$\lambda_1^{(d)}=o_{d\to\infty}(\log d)$, rigidity implies the product
condition, and the mixing time in total variation goes to
infinity by comparing with the projection of the process on an (affine)
eigenfunction associated to the spectral gap. This allows to use
\cite[Th.~3.3]{MR2375599} and \cite[Cor.~1]{salez-2025} to get cutoff for
$L^p$ distance, $1\leq p<\infty$, namely, denoting $\mu^{(d)}_t$ the law of
$X^{(d)}_t$, this writes
\begin{equation}\label{eq:Lp}
    \lim_{d\to\infty}
    \sup_{x^{(d)}_0\in S^{(d)}}
    \Bigl\|\frac{\mathrm{d}\mu^{(d)}_{t_d}}{\mathrm{d}\mu^{(d)}}-1\Bigr\|_{\mathrm{L}^p(\mu^{(d)})}
    =
    \begin{cases}
      \max & \text{if $t_d=(1-\varepsilon)t_*$}\\
      0 & \text{if $t_d=(1+\varepsilon)t_*$}
    \end{cases}
\end{equation}
where $\max=2$ if $p=1$ and $\max=+\infty$ if $1<p<\infty$, see \cite{blog-saloff}
for a pedagogical presentation. The case $p=\infty$ is excluded: it fails already for the Ornstein--Uhlenbeck process. We emphasize that
Corollaries~\ref{cor_cutoff_langevin_general},~\ref{co:cutoff:tv:H},
and~\ref{cor_convex_inter} give the mixing time as well as the cutoff.

The product condition can be seen as weak rigidity, connecting trend to equilibrium and spectral gap. 
It is well known that in the discrete setting, the product condition does not always imply cutoff, see for instance \cite{MR2375599} for detailed counter examples. Discrete state space processes can be more subtle than random walks on locally Euclidean graphs, which are the discrete analogue of diffusions on manifolds. In particular they can allow complete graphs moves producing quick convergence to equilibrium. For diffusions, the analogue of this type of degeneracy would be a non-local term. 

In another direction, for $\mathrm{e}^{-V}$ log-concave on $\mathbb{R}^d$, the
Kannan, Lovász and Simonovitz (KLS) conjecture states that the inverse
spectral gap $\lambda_1^{-1}$ of $\Delta-\nabla V\cdot\nabla$ is up to
universal multiplicative constants equal to the spectral radius $r(K)$ of the
covariance matrix
\[
K_{ij}=\int x_ix_j\mathrm{d}\mu(x)-m_im_j
\quad\text{where}\quad
m_i=\int x_i\mathrm{d}\mu(x).
\]
The best available upper bound is $\lambda_1^{-1}\leq c_d r(K)$ with $c_d$
logarithmic in $d$, see \cite{klartag}. The reverse bound
$r(K)\leq\lambda_1^{-1}$ is always available.

\subsection{Positively curved diffusions, curvature product condition, and cutoff}

Recall that the recent result \cite[Cor.~1]{salez-2025} states that for
non-negatively curved diffusions, and for total variation distance, the
product condition implies cutoff at the mixing time (a direct corollary of a
theorem bounding the mixing time window). By combining the approaches used in
the proof of \cite[Cor.~1]{salez-2025}, the proof of our
Corollary~\ref{co:cutoff:tv:H}, and the proof of \cite[Th.~3.3]{MR2375599}, we
arrive at the following theorem. We state it in the restricted Euclidean
setting for simplicity, but it holds in the Riemannian setting, since it is
the curvature bound and the diffusion property that matter.

\begin{theorem}[Cutoff for positively curved diffusions with curvature product condition]\label{th:prodcondcut}
  Let $(X_t^{(d)})$ be a sequence of diffusion processes on $\mathbb{R}^d$ of
  the form \eqref{eq:eds}, with convex potential $V^{(d)}$, initial condition
  $x_0^{(d)}$, and invariant measure $\mu^{(d)}$. Let
  $S^{(d)}\subset\mathbb{R}^d$ be an arbitrary non-empty set of initial
  conditions. Let $\eta\in(0,1)$ be an arbitrary fixed threshold which does
  not depend on $d$. Suppose that there exists positive constants
  $\kappa^{(d)}$ such that $\mathrm{Hess}(V^{(d)})\geq\kappa^{(d)}I_d$ and
  that the following curvature product condition holds:
  \[
    \lim_{d\to\infty}\kappa^{(d)}t_0^{(d)}=+\infty
    \quad\text{where}\quad
    t_0^{(d)} := \inf\bigl\{t \geq 0:\sup_{x_0^{(d)} \in S^{(d)}} \DTV(X_t^{(d)}, \mu^{(d)}) \leq\eta\bigr\}.
  \]
  Assume also that $t_0^{(d)}$ is finite for all sufficiently large $d$.
  Then there is cutoff at the critical time $t_0^{(d)}$ in the sense that for all $\varepsilon\in(0,1)$, 
  \begin{align}
    \lim_{d\to\infty}
    \sup_{x^{(d)}_0\in S^{(d)}}
    \DTV(X^{(d)}_{t_d},\mu^{(d)})
    &=
      \begin{cases}
        1 & \text{if $t_d=(1-\varepsilon)t_0^{(d)}$}\\
        0 & \text{if $t_d=(1+\varepsilon)t_0^{(d)}$}
      \end{cases}\\     
    \lim_{d\to\infty}
    \sup_{x^{(d)}_0\in S^{(d)}}
    \mathrm{H}(X^{(d)}_{t_d}\mid\mu^{(d)})
    &=
      \begin{cases}
        +\infty & \text{if $t_d=(1-\varepsilon)t_0^{(d)}$}\\
        0 & \text{if $t_d=(1+\varepsilon)t_0^{(d)}$}
      \end{cases}
  \end{align}
  and, under the stronger assumptions $\varliminf_{d\to\infty}\kappa^{(d)}>0$ and $\lim_{d\to\infty}t_0^{(d)}=+\infty$, we also have
  \begin{align}
    \lim_{d\to\infty} \sup_{x^{(d)}_0\in S^{(d)}} \mathrm{I}(X^{(d)}_{t_d}\mid\mu^{(d)})
    &=
      \begin{cases}
        +\infty & \text{if $t_d=(1-\varepsilon)t_0^{(d)}$}\\
        0 & \text{if $t_d=(1+\varepsilon)t_0^{(d)}$}
      \end{cases}\\
    \lim_{d\to\infty}
    \sup_{x^{(d)}_0\in S^{(d)}}
    \mathrm{W}_2(X^{(d)}_{t_d},\mu^{(d)})
    &=
      \begin{cases}
        +\infty & \text{if $t_d=(1-\varepsilon)t_0^{(d)}$}\\
        0 & \text{if $t_d=(1+\varepsilon)t_0^{(d)}$}
      \end{cases}.
  \end{align}
\end{theorem}

\subsection{Rigidity and Gaussian factorization for curved diffusions}
\label{ss:rigidity}

It turns out that the presence of a Gaussian factor, as well as the
$\lambda_1$ uniform lower bound on the Hessian of the potential
\eqref{eq:curgap}, are both equivalent to a notion of rigidity.

Let us start with the notion of Bakry--Émery curvature. For a Langevin
process in the Euclidean space as in \eqref{eq:eds}, for all $\rho>0$,
the following properties are equivalent:
\begin{enumerate}[label=(C\arabic*)]
\item $V-\frac{\rho}{2}\left|\cdot\right|^2$ is convex, in other words $V$ is $\rho$-convex 
\item $\mathrm{Hess}(V)(x)\geq\rho\mathrm{I}_d$ as quadratic forms, for all
  $x\in\mathbb{R}^d$
\item $\langle\mathrm{Hess}(V)\nabla f,\nabla f\rangle\geq\rho|\nabla f|^2$
  for all smooth $f:\mathbb{R}^d\to\mathbb{R}$.
\item the curvature-dimension inequality
$\mathrm{CD}(\rho,\infty)$ is satisfied, see Section~\ref{ss:riemann}.
\end{enumerate}
Following \cite{BGL14}, when these properties hold true, we say that the process has (Bakry--Émery) curvature bounded
below by $\rho$. It follows then that
\begin{equation}\label{eq:Lichne}
  \lambda_1\geq\rho.
\end{equation}
To see it, we observe that from the Bochner formula
$\nabla\mathcal{L}-\mathcal{L}\nabla=-\mathrm{Hess}(V)\nabla$, we get
\begin{equation}
  \int\bigl(\langle\mathrm{Hess}(V)\nabla f,\nabla f\rangle-\rho|\nabla f|^2\bigr)\mathrm{d}\mu
  =\int\bigl(\langle\mathcal{L}\nabla f-\nabla\mathcal{L} f,\nabla f\rangle-\rho|\nabla f|^2\bigr)\mathrm{d}\mu,
\end{equation}
while on the other hand, by integration by parts, we get\footnote{We denote by $\left\|\cdot\right\|_{\mathrm{HS}}$ the Hilbert--Schmidt or
trace or Frobenius matrix norm.}
\begin{equation}
  -\int\langle\mathcal{L}\nabla f,\nabla f\rangle\mathrm{d}\mu
  =\sum_{i=1}^d\int|\nabla\partial_if|^2\mathrm{d}\mu
  =\sum_{i,j=1}^d\int(\partial^2_{ij}f)^2\mathrm{d}\mu
  =\int\|\mathrm{Hess}(f)\|_{\mathrm{HS}}^2\mathrm{d}\mu,
\end{equation}
therefore, by specializing to $f$ such that $-\mathcal{L} f=\lambda_1f$, we obtain
\begin{equation}\label{eq:R}
  \int\bigl(\underbrace{\langle\mathrm{Hess}(V)\nabla f,\nabla f\rangle-\rho|\nabla f|^2}_{\geq0\text{ by (C3)}}\bigr)
  \mathrm{d}\mu
  =(\lambda_1-\rho)\int|\nabla f|^2\mathrm{d}\mu
  -\int\|\mathrm{Hess}(f)\|_{\mathrm{HS}}^2\mathrm{d}\mu,
\end{equation}
which gives \eqref{eq:Lichne}\footnote{We could also invoke the
  Brascamp--Lieb inequality
  $\int f^2\mathrm{d}\mu-(\int
  f\mathrm{d}\mu)^2\leq\int\langle\mathrm{Hess}(V)^{-1}\nabla f,\nabla
  f\rangle\mathrm{d}\mu$, or equivalently the covariance representation of
  Hörmander--Helffer--Sjöstrand.}. Equality in \eqref{eq:Lichne} is
obviously achieved for the OU case $V=\frac{\rho}{2}\left|\cdot\right|^2$.
More generally, beyond the OU case, we say that the process is \emph{rigid}
when the spectral gap matches the curvature lower bound: $\lambda_1=\rho$.
Rigidity can be reformulated, and it turns out that for all $\rho>0$, the
following items are equivalent:
\begin{enumerate}[label=(R\arabic*)]
\item The process is rigid: it has curvature bounded below by $\rho$ and $\lambda_1=\rho$
\item Up to a rotation and translation (change of coordinates), the
  probability measure $\mu=\mathrm{e}^{-V}$ is product, and splits into
  a one-dimensional Gaussian factor $\mathcal{N}(0,\frac{1}{\rho})$ and a
  second factor which is log-concave with a $\rho$-convex potential. In which
  case, we get $\lambda_1=\rho$ and all the eigenfunctions associated to
  $\lambda_1$ are affine.
\end{enumerate}
It is immediate to check that (R2) gives (R1). Conversely, if (R1) is
satisfied, then for $f$ such that $-\mathcal{L} f=\lambda_1f$, we obtain that both
sides of \eqref{eq:R} are identically zero. This gives
$\mathrm{Hess}(f)\equiv0$, hence $f$ is affine, and this also gives that the
constant vector $\nabla f$ is in the kernel of the symmetric positive
semidefinite matrix $\mathrm{Hess}(V)-\rho\mathrm{I}_d$. Moreover
$\nabla f\neq0$ since $f$ is not constant, leading to (R2). We observe that by
iterating the procedure, we get a Gaussian factor with same dimension as the
eigenspace of $\lambda_1$.

The equivalence between the rigidity properties (R1) and (R2) appears in
\cite{CZ17} in the traditional broader Riemannian setting of Bakry--Émery,
discussed in Section~\ref{se:riemann}. An alternative proof in the Euclidean
case, based on optimal transport, was found by \cite{DPF17}. These works were
motivated by the more general problem of understanding manifolds optimizing a
certain geometric quantity (namely, the spectral gap) under a curvature
constraint. This type of question goes back to a theorem of
\cite{zbMATH03188146}, which states that positively curved smooth Riemannian
manifolds with minimal spectral gap are isometric to spheres. Rigidity is
typically used to reinforce statements in comparison geometry. See for example
\cite{zbMATH01545688} for applications of Gamma calculus to comparison
geometry under curvature assumptions.

\subsection{Riemannian manifolds}\label{ss:riemann}

We now discuss Theorem~\ref{th:maineucl} in a broader geometric context. A
weighted Riemannian manifold is a triplet $(M, g, \mu)$, where $(M, g)$ is a
Riemannian manifold with metric tensor $g$, and $\mu$ is a measure on $M$.
Throughout the Riemannian setting, we assume that $(M,g)$ is smooth,
connected, complete, and without boundary.
Here we shall assume that $\mu$ is a probability measure, absolutely
continuous with respect to the volume measure. We shall write
$\mu =\mathrm{e}^{-V}$, so that $V$ plays the same role of a potential as in
the Euclidean setting. The Markov process ${(X_t)}_{t\geq0}$ we consider is
the drift diffusion that combines a Brownian motion ${(B_t)}_{t \geq 0}$ on
$(M, g)$ and drift $-\nabla V$, given by the SDE
\begin{equation} \label{eq_sde_riem}
  \mathrm{d}X_t = -\nabla V(X_t)\mathrm{d}t + \sqrt{2}\mathrm{d}B_t,
  \quad X_0=x_0\in M,
\end{equation}
where $\nabla$ is the gradient on $(M,g)$. We refer to \cite{BGL14} and \cite{Hsu} for
more background. A key role is played by the Ricci curvature tensor of the
manifold, which we shall denote by $\operatorname{Ric}$. We refer to
\cite{Oll13} for an introduction to curvature on Riemannian manifolds. Under
the assumption that the Ricci curvature is bounded from below, and that $V$ is
geodesically semi-convex, then solutions to \eqref{eq_sde_riem}
exist for all times \cite[Th.~11.8]{Grig}. We shall not make use of 
\eqref{eq_sde_riem}, and only rely on the generator of the process
\begin{equation}\label{eq:L:riem}
\mathcal{L} = \Delta - \nabla V \cdot \nabla
\end{equation}
where $\Delta$ is the Laplace--Beltrami operator on $(M,g)$. The weighted
manifold $(M, g, \mu)$ has Bakry--Émery curvature bounded below by $\rho$
when
\begin{equation}\label{eq:BErho}
  \operatorname{Ric} + \mathrm{Hess}(V) \geq \rho g.
\end{equation}
A celebrated result due to \cite{BE85}, see
also \cite{BGL14}, states that this implies a lower bound on the spectral gap
of $\mathcal{L}$ in the sense that
\begin{equation}
\lambda_1 \geq \rho. 
\end{equation}
When the manifold is unweighted, this is the dimension-free version of the
famous spectral gap bound due to \cite{Lic58}. Actually the Bakry--Émery
approach implies stronger functional inequalities, such as a logarithmic
Sobolev inequality, and a Gaussian isoperimetric inequality. Note that
following for instance \cite[Prop.~6.7]{GMS15}, when $\rho > 0$ the spectrum
of $\mathcal{L}$ is discrete and the spectral gap $\lambda_1$ is always an
eigenvalue of $-\mathcal{L}$. Moreover, the eigenfunctions associated to the
spectral gap are continuous.

In the Bakry--Émery approach, the bound \eqref{eq:BErho} is formulated in
a more abstract way involving the $\Gamma$ and $\Gamma_2$ functional quadratic
forms, defined by 
\begin{align}
  \Gamma(f,g)&=\tfrac{1}{2}(\mathcal{L}(fg)-f\mathcal{L} g-g\mathcal{L} f)\\
  \Gamma_2(f,g)&=\tfrac{1}{2}(\mathcal{L}\Gamma(f,g)-\Gamma(f,\mathcal{L} g)-\Gamma(g,\mathcal{L} f)).
\end{align}
They are fully characterized by their diagonal $\Gamma(f)=\Gamma(f,f)$ and
$\Gamma_2(f)=\Gamma_2(f,f)$, by polarization. In the case of the Langevin
operator \eqref{eq:L:riem}, they simply boil down to
\begin{equation}
  \Gamma(f)=|\nabla f|^2
  \quad\text{and}\quad
  \Gamma_2(f)=\mathrm{Ric}(\nabla f,\nabla f)+\|\mathrm{Hess}(f)\|_{\mathrm{HS}}^2+\langle\mathrm{Hess}(V)\nabla f, \nabla f \rangle,
\end{equation}
The formula for $\Gamma_2$ follows from the Bochner formula
$\nabla\mathcal{L}=\mathcal{L}\nabla-\mathrm{Hess}(V)\nabla-\mathrm{Ric}(\nabla,\nabla)$.
The Bakry--Émery curvature-dimension inequality $\mathrm{CD}(\rho,n)$,
where $\rho\in\mathbb{R}$ is the curvature and $n\in\mathbb{R}$ is the
dimension, writes, for all $f$ in a rich enough class of functions,
\begin{equation}
  \Gamma_2(f)\geq\rho\Gamma(f)+\frac{1}{n}(\mathcal{L} f)^2.
\end{equation}
Now $\mathrm{CD}(\rho,\infty)$ is equivalent to the curvature lower bound
\eqref{eq:BErho}. Moreover, in the Euclidean case, the Ricci tensor is zero
and $\mathrm{CD}(\rho,\infty)$ is equivalent to say that $V$ is $\rho$-convex.
 
In order to formulate a neat abstract analogue of Theorem~\ref{th:maineucl},
we shall use vector-valued functions whose coordinates are eigenfunctions, as
defined in the following:

\begin{definition}[Multi-eigenfunction]\label{df:multeig}
  Let $E_1$ be the eigenspace of $-\mathcal{L}$ associated with the eigenvalue
  $\lambda_1$, in $L^2(\mu)$. Let $k_1 = \dim{E_1}$. We define a
  multi-eigenfunction as being a map $M \to\mathbb{R}^{k_1}$ whose coordinates are
  orthogonal elements of $E_1$ in $L^2(\mu)$. We denote by $F_1$ the set of
  all multi-eigenfunctions, and by $\mathbb{S}_{F_1}$ the set of multi-eigenfunctions
  whose coordinates are elements of the unit sphere of $E_1$ in $L^2(\mu)$,
  thus orthonormal.
\end{definition}

The link between this definition and Theorem~\ref{th:maineucl} is that the
component $\sum_{i=1}^{k_1}|f_i(x_0)|^2$ in the definition of $\Lambda$ in
Theorem~\ref{th:maineucl} is exactly the $\ell^2$ norm of $T(x_0)$ where $T$
is some element of $\mathbb{S}_{F_1}$. The Riemannian analog of
Theorem~\ref{th:maineucl} is now the following.

\begin{theorem}[Wasserstein estimate on weighted Riemannian manifold]\label{main_thm_riem}
  Let ${(X_t)}_{t\geq0}$ be the diffusion \eqref{eq_sde_riem} on a smooth,
  connected, complete Riemannian manifold without boundary, and let
  $\mathcal{L}$ be its generator \eqref{eq:L:riem}. Assume that for some
  $\rho>0$,
  \begin{equation}
   \text{$\mathrm{CD}(\rho, \infty)$ is satisfied and $\lambda_1 = \rho$}.
  \end{equation}
  Then for all non-empty set of possible initial conditions $S\subset M$, and all $t\geq0$, 
  \begin{equation} 
    \frac{\mathrm{e}^{-2\rho t}}{\rho}\sup_{x_0 \in S}\Lambda(x_0) %
     +\frac{k_1\bigl(1-\sqrt{1-\mathrm{e}^{-2\rho t}}\bigr)^2}{\rho} %
    \leq \sup_{x_0\in S}\mathrm{W}_2(X_t, \mu)^2 %
    \leq \mathrm{e}^{-2\rho t} \sup_{x_0 \in S}\int d_M(x_0, x)^2\mathrm{d}\mu(x),
  \end{equation}
  where $d_M$ is the Riemannian distance on $M$, and where
  \begin{equation}
    \Lambda(x)=\sup_{T\in\mathbb{S}_{F_1}}|T(x)|^2.
  \end{equation}
  Moreover, if $S=\{x\in M:d_M(x,m)\leq R\}$ with $\Lambda(m)=0$ then we can
  replace $\sup_{x_0\in S}\Lambda(x_0)$ by $\rho R^2$.
\end{theorem}

\subsection{About stability}

Our study is under the strong assumption $\lambda_1 = \rho$. For the
$d$-dimensional unit sphere, we have $\rho = d-1$ and $\lambda_1 = d$, so the
abstract theorem does not apply. Yet cutoff does occur, and moreover
$\lambda_1/\rho \longrightarrow 1$. This naturally leads to asking

\begin{question}
  Is there a (useful) analog of Corollary~\ref{cor_cutoff_langevin_general}
  under the weaker assumption
  \begin{equation}
    \frac{\lambda_1^{(d)}}{\rho^{(d)}} 
    \xrightarrow[d\to\infty]{}1\ ?
  \end{equation}
\end{question}

Some results on positively curved manifolds when $\lambda_1$ is close to
$\rho$ were obtained by \cite{CF20,MO21,BF22,FGS24}. However, the quantitative
estimates on eigenfunctions seem to be too weak to easily generalize
Theorem~\ref{main_thm_riem}. The estimate on the Wasserstein distance strongly
relies on the fact that the eigenfunction is Lipschitz, and the estimate is
sharp because it is actually affine, namely that $\mathrm{Hess}(f) = 0$. When
$\lambda_1$ is close to $\rho$, then $\mathrm{Hess}(f)$ is small, but for
example the smallness estimates in \cite{BF22} are only in $L^2$ norm, so we
do not actually control the Lipschitz norm of $f$. It is unclear if a stronger
$L^\infty$ estimate can be expected in general, so maybe some extra
assumptions are needed.

\subsection*{Acknowledgments}
This work has been (partially) supported by the Project CONVIVIALITY ANR-23-CE40-0003 of the French National Research Agency.    

\section{Further comments}

This section gathers several remarks, mostly about the Euclidean case for simplicity.

\subsection{Fisher information and the heat kernel}

It is natural to try to directly upper bound the Fisher information by using
the heat kernel, instead of using a comparison with relative entropy.
Following \cite{BCL23}, let us examine first the exactly solvable case of the
OU process, for which the Mehler formula
$\mu_t=\mathcal{N}(\mathrm{e}^{-\rho t}x_0,\frac{1-\mathrm{e}^{-2\rho
    t}}{\rho}\mathrm{I}_d)$ gives, for $t>0$,
\begin{equation}
  \nabla\log\frac{\mathrm{d}\mu_t}{\mathrm{d}\mu}(x)
  =-\rho\frac{x-\mathrm{e}^{-\rho t}x_0}{1-\mathrm{e}^{-2\rho t}}+\rho x
  =a_tx+b_tx_0
\end{equation}
where $a_t$ and $b_t$ do not depend on the dimension $d$, leading to
\begin{equation}
  \mathrm{I}(\mu_{t}\mid\mu)
  =\rho d\frac{\mathrm{e}^{-4\rho t}}{1-\mathrm{e}^{-2\rho t}}+\rho^2\mathrm{e}^{-2\rho t}|x_0|^2.
\end{equation}
This gives the upper bound \eqref{eq:I:upper} as well as the lower bound
\eqref{eq:I:lower}. Beyond the OU process, denoting, for all $t>0$, by
$p_{t,x_0}$ the density of $\mu_t$ with respect to the Lebesgue measure, which
is known as the heat kernel of $\mathcal{L}$, we have, for all $x$ and all
$t>0$,
\begin{equation}
  \nabla\log\frac{\mathrm{d}\mu_t}{\mathrm{d}\mu}(x)
  =\nabla\log p_{t,x_0}(x)+\nabla V(x).
\end{equation}
Inspired by the OU example, we could expect that under $\mathrm{CD}(\rho,\infty)$, for $t$ small enough,
\begin{equation}\label{eq:loggradbound}
  \bigl|\nabla\log p_{t,x_0}(x)\bigr|^2
  \leq a_t^2|x|^2+b_t^2|x_0|^2,
\end{equation}
which recalls the parabolic logarithmic Harnack and Li--Yau inequalities, leading to 
\begin{equation}
  \mathrm{I}(\mu_{t}\mid\mu)
  \leq 2\int(a_{t}^2|x|^2+|\nabla V(x)|^2)\mathrm{d}\mu_t(x)+2b_{t}^2|x_0|^2.
\end{equation}
Here $a_t$ and $b_t$ are constants with respect to $x$, that may depend on $t$
but not on $d$ and $x_0$. The explicit dependence on the dimension and initial
condition is crucial here for our purposes. Unfortunately, we were not able to
locate \eqref{eq:loggradbound} in the literature on short time logarithmic
gradient bounds for the heat kernel on non-compact manifolds with a lower
curvature bound. Close works are \cite{MR2253731} where a constant is not
controlled with respect to the dimension, and \cite{MR4546623} which is
curvature free but is restricted to a compact subset. It would not be a
surprise to assume, in addition to positive curvature, an additional condition
regarding eigenfunctions associated to the spectral gap.

To summarize, the upper bound on $\mathrm{I}(\mu_{t}\mid\mu)$ follows from a
logarithmic gradient bound for the heat kernel in short time, which expresses
a regularization property of the evolution equation driven by $\mathcal{L}$ when it
starts from a Dirac mass. Such an upper bound on the Fisher information
$\mathrm{I}$ implies all the others. Indeed, it would give an upper bound on
$\mathrm{H}$ by using the logarithmic Sobolev inequality, giving in turn an
upper bound on $\mathrm{W}_2$ by using the Talagrand inequality, as well as an
upper bound on $\mathrm{d}_{\mathrm{TV}}$ by using the Pinsker inequality.

An estimate on $\mathrm{I}$ is at the heart of the varentropy approach in
\cite{Salez24}. In \cite{salez-2025}, the control of Fisher information is
circumvented by a direct comparison with varentropy.


\subsection{Optimal constants}

An observation that dates back to \cite{BE85}, see also \cite{BGL14}, is that $\mathrm{CD}(\rho,\infty)$ implies a
logarithmic Sobolev inequality (LSI) with constant $2/\rho$ as well as a
Poincaré inequality (PI) of constant $1/\rho$ (directly as well as by
linearization). Recall that the optimal Poincaré constant is precisely the
inverse of the spectral gap $1/\lambda_1$. It follows that rigidity implies
that the optimal LSI constant is twice the optimal PI constant, just like for
the OU process. By the way, a famous alternative due to Oscar Rothaus states
that in the compact setting, if the optimal LSI constant is not twice the
optimal PI constant, then there exists an extremal function for LSI, see for
instance \cite[Th.~at bottom of p.~107]{MR620581} as well as
\cite[Th.~2.2.3 p.~333]{MR1490046} in the lecture notes by Laurent
Saloff-Coste for a discrete version. Note also that in the case \eqref{eq:VW},
the function $x\mapsto\mathrm{e}^{\alpha(x_1+\cdots+x_d)}$ is extremal for
LSI, as observed in \cite{MR4175749}, while $x\mapsto x_1+\cdots+x_d$ is
extremal for PI.

\subsection{Normalization}\label{ss:normalization}

For any real $\alpha>0$ which may depend on $d$, the time-changed process
$X^{(\alpha)}:={(X_{\alpha t})}_{t\geq0}$ solves 
$\mathrm{d}X^{(\alpha)}_t=-\alpha\nabla
V(X^{(\alpha)}_t)\mathrm{d}t+\sqrt{2\alpha}\mathrm{d}B_t$ and has generator
$\alpha\mathcal{L}$. The process $X$ has cutoff at critical time
$t_*$ iff $X^{(\alpha)}$ has cutoff at critical time $t_*/\alpha$.
Such scaled processes play a role with respect to mean-field limits of
interacting particle systems related to McKean--Vlasov semilinear PDE, see
\cite{BCL23}. Note that the product condition is invariant under such a scaling. 

\subsection{Temperature}

For any real $\sigma>0$ independent of $d$, playing the role of the
temperature, let us consider the Markov diffusion process
${(Y^{(\sigma)}_t)}_{t\geq0}$ solving the stochastic differential equation
$\mathrm{d}Y^{(\sigma)}_t=-\nabla
V(Y^{(\sigma)}_t)\mathrm{d}t+\sqrt{2\sigma^2}\mathrm{d}B_t$. This is
\eqref{eq:eds} when $\sigma=1$. Its invariant law is
$\mathrm{d}\mu_\sigma(x)=Z_\sigma^{-1}\mathrm{e}^{-V(x)/\sigma^2}\mathrm{d}x$,
where $Z_\sigma=\int_{\mathbb{R}^d}\mathrm{e}^{-V(x)/\sigma^2}\mathrm{d}x$,
while its generator is
$\sigma^2\Delta-\nabla V\cdot\nabla$. The Bakry--Émery operators are
 \begin{equation}
     \Gamma(f)=\sigma^2|\nabla f|^2
     \quad\text{and}\quad
     \Gamma_2(f)=\sigma^4\|\mathrm{Hess}f\|_{\mathrm{HS}}^2+\sigma^2\langle\mathrm{Hess}(V)\nabla f,\nabla f\rangle.
 \end{equation}
 Thus, for all $\rho>0$, $\mathrm{CD}(\rho,\infty)$ of $\mathcal{L}$ is equivalent to
 $\mathrm{Hess}(V)\geq\rho\mathrm{I}_d$ as quadratic form, which is free of
 $\sigma$. In particular, in the rigid case $\rho=\lambda_1$, the cutoff
 critical time is free of $\sigma$. This can be also obtained from
 Section~\ref{ss:normalization} by using $\alpha=\sigma^2$ and
 $\frac{1}{\sigma^2}V$ instead of $V$.

\section{Proofs}

\subsection{Preliminaries on rigidity}

As already mentioned, our starting point is a Gaussian splitting or
factorization theorem of \cite{CZ17}, which in the Euclidean setting takes the
form of Theorem~\ref{thm_split_eucl} below. See also Section~\ref{ss:rigidity}
and \cite[Rem.~3.2]{MR4175749}.

\begin{theorem}[Gaussian factorization in the Euclidean
  space] \label{thm_split_eucl} Consider the operator \eqref{eq:L} and assume
  that for some constant $\rho>0$, $V$ is $\rho$-convex and
  $\lambda_1 = \rho$. Then there is an orthonormal basis $(e_1,\ldots,e_d)$ of
  $\mathbb{R}^d$ and a vector $m \in \mathbb{R}^d$ such that $V$ is of the form
    \begin{equation}
    V(x) = \frac{\rho}{2}((x_1-m_1)^2 + \cdots + (x_{k_1}-m_{k_1})^2) + \widetilde{V}(x_{k_1+1},\ldots,x_d)
  \end{equation}
  where $k_1$ is the dimension of the eigenspace $E_1$ of $-\mathcal{L}$
  associated with $\lambda_1$, and $\widetilde{V}$ is $\rho$-convex on
  $\mathbb{R}^{d-k_1}$. Moreover, all eigenfunctions with eigenvalue
  $\lambda_1$ are affine, and only depend on the first $k_1$ coordinates in
  the above basis.
\end{theorem}

The vector $m$ actually is the center of mass of the probability measure $\mathrm{e}^{-V}$. 

\begin{remark}[Eigenfunctions structure] \label{rem_eigen_grad} By
  Theorem~\ref{thm_split_eucl}, if $f\in E_1$, then $f$ is affine, and since
  $\int |\nabla f|^2\mathrm{d}\mu = -\int (\mathcal{L} f)f\mathrm{d}\mu
  =\lambda_1\int f^2\mathrm{d}\mu$, we get
  $f(x)=\langle a, p(x)\rangle_{\mathbb{R}^{k_1}} + b$ with $p$ the projection
  on the $k_1$-dimensional Gaussian factor and
  $|a|^2 = \lambda_1\|f\|_{L^2(\mu)}^2$.
\end{remark}

Another way of stating this result is that up to a rotation (the change of
basis) and a translation by the vector $m$, the law $\mathrm{e}^{-V}$ is a
product measure, with a centered Gaussian factor with variance $1/\rho$ on the
first $k_1$ coordinates, and a log-concave factor with a $\rho$-convex
potential on the last $d-k_1$ coordinates. Note that the result can only be
true up to a rotation and translation, since the assumptions are stable by
isometries of $\mathbb{R}^d$.

An alternative proof in the Euclidean setting, based on a rigidity property
for regularity of solutions to the Monge--Ampère PDE for optimal transport
maps, was given by \cite{DPF17}. We shall discuss some elements of proof in
the full Riemannian setting in Section~\ref{se:riemann}.

\begin{lemma}[OU process associated to an eigenfunction] \label{lem_proj_ou}
  Under the setting of Theorem~\ref{thm_split_eucl} and
  Definition~\ref{df:multeig}, for all $T \in \mathbb{S}_{F_1}$, the process
  ${(T(X_t))}_{t\geq0}$ is a $k_1$-dimensional OU process scaled by a factor
  $\rho$, that is a process on $\mathbb{R}^{k_1}$ with generator
  $\rho \Delta - \rho x \cdot \nabla$.
\end{lemma}

\begin{proof}
  Let us show that $T(X)$ is a Markov process and compute its
  generator. 
  If $\vec{v}$ is a vector-valued function, then $\vec{\mathcal{L}} \vec{v} $
  is the vector obtained by applying $\mathcal{L}$ to each coordinate, and
  $\Gamma(\vec{v})$ the matrix whose coefficients are
  $\Gamma(\vec{v}_i, \vec{v}_j)$. Since the coordinates of $T$ are orthogonal
  normalized eigenfunctions, $\vec{\mathcal{L}} T = - \lambda_1 T$ and
  $\Gamma(T) = \rho \operatorname{Id}$, as per Theorem~\ref{thm_split_eucl}
  (affine with orthogonal constant gradients of norm $\sqrt{\rho}$). By the
  diffusion property, we have
  \begin{align}
    \mathcal{L}(g \circ T) &= \nabla g \circ T \cdot \vec{\mathcal{L}} T + \langle \nabla^2 g \circ T, \Gamma(T)\rangle \\
                  &= - \lambda_1 \nabla g\circ T \cdot T + \rho \Delta g \circ T \\
                  &= - \rho \nabla g\circ T \cdot T + \rho \Delta g \circ T.
  \end{align}
  This is a function of $T$, so $T(X)$ is a Markov process, and when viewing
  it as such it is indeed equal to $\rho$ times the generator of an OU process
  with unit variance, applied to a function $g$.
\end{proof}

Hence, in the rigid case, the full process contains an OU subprocess, and
hence cannot converge to equilibrium faster than it. This will yield the lower
bound in Theorem~\ref{th:maineucl}.

\subsection{Proof of Theorem~\ref{th:maineucl}}

The upper bound is an immediate well-known consequence of the exponential
convergence to equilibrium in Wasserstein distance under the
CD$(\rho=\lambda_1, \infty)$ condition, see for instance \cite{vRS05}: for all
$X_0,Y_0$ and $t\geq0$,
\begin{equation}
  \mathrm{W}_2(X_t, Y_t) 
  \leq \mathrm{e}^{-\rho t}\mathrm{W}_2(X_0, Y_0).
\end{equation}
The desired upper bound follows by taking $Y_0\sim\mu$ and $X_0$
deterministic.

Let us now prove the lower bound. Consider a multi-eigenmap $T$, see
Definition~\ref{df:multeig}. Up to an isometry, it is of the form
$T(x) = Ax + b$. Moreover, since the coordinates of $T$ are orthogonal
eigenfunctions, the rows of $A$ are orthogonal vectors. If moreover we assume
that $T \in \mathbb{S}_{F_1}$, then each row of $A$ has norm
$\sqrt{\rho} = \sqrt{\lambda_1}$. Hence we get
\begin{equation}
  \mathrm{W}_2(T(X_t), T(Y_t)) \leq \sqrt{\lambda_1}\mathrm{W}_2(X_t,Y_t).
\end{equation}

Applying Lemma~\ref{lem_proj_ou}, $Z_t:=T(X_t)$ and $T(Y_t)$ are OU processes
with speed accelerated by a factor $\rho$. Suppose now that $Y_0$ is
distributed according to the equilibrium measure: $Y_0\sim\mu$. Then $T(Y_t)$
follows a standard Gaussian law $\gamma_{k_1}$ on $\mathbb{R}^{k_1}$, for all
$t$. Thus using the formula
$\mathrm{W}_2(\mathcal{N}(m_1,\Sigma_1),\mathcal{N}(m_2,\Sigma_2))^2=|m_1-m_2|^2+\mathrm{Tr}((\sqrt{\Sigma_1}-\sqrt{\Sigma_2})^2)$
when $\Sigma_1\Sigma_2=\Sigma_2\Sigma_1$, the Mehler formula
$Z_t\sim\mathcal{N}(T(x_0)\mathrm{e}^{-\rho t},(1-\mathrm{e}^{-2\rho
  t})I_{k_1})$, and $\gamma_{k_1}=\mathcal{N}(0,I_{k_1})$, we find
\begin{equation}
  \mathrm{W}_2(Z_t, \gamma_{k_1})^2 
  = \mathrm{e}^{-2\rho t}|T(x_0)|^2+k_1\bigl(1-\sqrt{1-\mathrm{e}^{-2\rho
      t}}\bigr)^2.
\end{equation}
Optimizing over $x_0$ and $T$ concludes the proof. 

We now consider the case where $\mu$ is centered and $S = B(0,R)$. For any
$T \in \mathbb{S}_{F_1}$, writing $T = Ax+b$, since $\int T\mathrm{d}\mu = 0$ we see
that $b = 0$. Therefore $T(0) = 0$. Moreover, since the rows of $A$ are
orthogonal and have norm $\sqrt{\lambda_1}$, we get
$\sup_{x\in S}|Ax|^2 = \lambda_1 R^2$.

\subsection{Proof of Corollary~\ref{cor_cutoff_langevin_general}}

Let us drop the $(d)$ superscript. Without loss of generality, we can assume
that $m = 0$ by translating $V$. We start by using the lower bound
of Theorem~\ref{th:maineucl} to prove the convergence to infinity when
$t_d = (1-\epsilon)t_*$. Since we are in the centered setting and the set of
initial conditions is a centered ball of radius $c\sqrt{d}$, the lower bound
\begin{equation}
  \frac{\mathrm{e}^{-2\lambda_1 t}}{\lambda_1}\sup_{x_0\in S}\Lambda(x_0) 
  \leq \sup_{x_0\in S}\mathrm{W}_2(X_t, \mu)^2
\end{equation}
while $\sup_{x_0\in S}\Lambda(x_0)=\lambda_1c^2d$ thanks to
Theorem~\ref{th:maineucl} since $S=B(0,c\sqrt d)$ with $c>0$. Evaluating at
$t = t_d$ and neglecting $k_1$, we get
\begin{equation}
  cd^{\epsilon/2} \leq \sup_{x_0\in S}\mathrm{W}_2(X_{t_d}, \mu)
\end{equation}
and letting $d$ go to infinity concludes the proof of the lower bound. Let us
now prove the case $t_d = (1+\epsilon)t_*$ via the upper bound in
Theorem~\ref{th:maineucl}. Since $\mu$ is centered,
\begin{equation}
  \int|x-x_0|^2\mathrm{d}\mu(x) = |x_0|^2 + \int|x|^2\mathrm{d}\mu(x).
\end{equation}
From the spectral gap, for any centered $f$ we have the Poincaré inequality
\begin{equation}
  \int f^2\mathrm{d}\mu \leq \frac{1}{\lambda_1}\int|\nabla f|^2\mathrm{d}\mu.
\end{equation}
Applying this inequality to each coordinate yields
\begin{equation}
  \int|x|^2\mathrm{d}\mu \leq \frac{d}{\lambda_1}.
\end{equation}
Therefore we have the upper bound
\begin{equation}
  \sup_{x_0\in S}\mathrm{W}_2(X_t, \mu) 
  \leq \mathrm{e}^{-\lambda_1 t}\Bigl(\sup_{x_0\in S} |x_0|^2 + \frac{d}{\lambda_1}\Bigr)^{1/2}.
\end{equation}
Since $\sup_{x_0\in S} |x_0|^2 = c^2d$, evaluating at $t_d = (1+\epsilon)t_*$ yields
\begin{equation}
  \sup_{x_0\in S}
  \mathrm{W}_2(X_{t_d}, \mu) 
  \leq d^{-\epsilon/2}(c^2+(\lambda_1)^{-1})^{1/2}.
\end{equation}
Since $\liminf_{d\to\infty}\lambda_1 > 0$, letting $d$ go to infinity concludes the proof. 

\begin{remark}[Weaker assumption]
  In the proof, the assumption $\liminf_{d\to\infty}\lambda_1 > 0$ could have
  been replaced by a slow enough growth of $(\lambda_1)^{-1}$, for instance
  $(\lambda_1)^{-1} \leq (\log d)^\alpha$ for some $\alpha > 0$.
\end{remark}

\subsection{Proof of Corollary~\ref{co:cutoff:tv:H}}

Let us prove first the cutoffs in relative entropy and in total variation
distance by using Corollary~\ref{cor_cutoff_langevin_general} on the
Wasserstein distance. We drop the $(d)$ superscript.

\emph{Upper bound.} It is well known that CD$(\rho, \infty)$ with $\rho>0$
implies that for all $t \geq t_0>0$,
\begin{equation}\label{eq:expdecH}
  \mathrm{H}(X_t\mid\mu) \leq\mathrm{e}^{-2\rho(t-t_0)}\mathrm{H}(X_{t_0}\mid\mu).
\end{equation}
The role of $t_0>0$ is to benefit from the fact that the law of $X_{t_0}$ is
absolutely continuous with respect to $\mu$, while the law of $X_0$ is a Dirac
mass. On the other hand, following for instance
\cite[Lem.~4.2]{BGL01}\footnote{In \cite{BGL01}, it is at the heart of a
  derivation of the Otto--Villani $\mathrm{H}\mathrm{W}\mathrm{I}$
  inequality, which contains the logarithmic Sobolev inequality as well as the
  Talagrand transportation inequality.}, CD$(\rho, \infty)$ with $\rho>0$
implies\footnote{The first inequality is valid for all $\rho\in\mathbb{R}$.}
that for all $t > 0$,
\begin{equation}\label{eq:HWsmallt}
  \mathrm{H}(X_t\mid\mu) \leq \frac{\rho\mathrm{e}^{-2\rho t}}{1-\mathrm{e}^{-2\rho t}}\mathrm{W}_2(X_0, \mu)^2
  \leq\frac{1}{2t}\mathrm{W}_2(X_0, \mu)^2.
\end{equation}
Using \eqref{eq:HWsmallt} with $t=1$, and combining with \eqref{eq:expdecH}, we obtain, for all $t\geq1$,
\begin{equation}\label{eq:HtWX0}
  2\mathrm{H}(X_t\mid\mu) \leq \mathrm{e}^{-2\rho(t-1)}\mathrm{W}_2(X_0, \mu)^2.
\end{equation}
By combining with the general Csiszár--Kullback--Pinsker inequality
\begin{equation}\label{eq:CKP}
  \DTV(\mu,\nu)^2\leq\frac{1}{2}\mathrm{H}(\nu\mid\mu),
\end{equation}
we get finally, for all $t\geq1$,
\begin{equation}
  \DTV(X_t,\mu)^2\leq\frac{1}{2}\mathrm{H}(X_t\mid\mu)
  \leq \frac{1}{4}\mathrm{e}^{-2\rho(t-1)}\mathrm{W}_2(X_0, \mu)^2.
\end{equation}
Recall that $\rho=\lambda_1$. Now, proceeding as in the proof of Corollary
\ref{cor_cutoff_langevin_general}, we get, for all $c>0$,
\begin{equation}\label{eq:WX0upper}
  \sup_{x_0\in B(m,c\sqrt{d})}\mathrm{W}_2(X_0,\mu)^2=\sup_{x_0\in B(m,c\sqrt{d})}\int|x-x_0|^2\mathrm{d}\mu(x)
  \leq(c^2+\lambda_1^{-1})d.
\end{equation}
Using the definition of $t_*$ and the assumption
$\lambda_1=o_{d\to\infty}(\log(d))$, we get, for all $c,\varepsilon>0$,
\begin{equation}
  \sup_{x_0\in B(m,c\sqrt{d})}\mathrm{e}^{-2\rho((1+\varepsilon)t_*-1)}\mathrm{W}_2(X_0, \mu)^2
  \leq\mathrm{e}^{O(\lambda_1)}d^{-\varepsilon}\xrightarrow[d\to\infty]{}0.
\end{equation}
Therefore, going back to $\mathrm{H}$ and $\DTV$, we get, for all $c,\varepsilon>0$,
\begin{equation}
  \lim_{d\to\infty}\sup_{x_0\in B(m,c\sqrt{d})}\mathrm{H}(X_{(1+\epsilon)t_*}\mid\mu)=0 
  \ \text{and}\ %
  \lim_{d\to\infty}\sup_{x_0\in B(m,c\sqrt{d})}\DTV(X_{(1+\epsilon)t_*},\mu)=0.
\end{equation}
The approach differs from the one in \cite{BCL23}, in the way we regularize in
\eqref{eq:HWsmallt} as well as in the way we control relative entropy, here
via the Wasserstein distance.

\emph{Lower bound}. Both total variation distance and relative entropy
decrease by mappings,
\begin{equation}
  \DTV(\mu \circ T^{-1}, \nu \circ T^{-1}) \leq \DTV(\mu,\nu)
  \quad\text{and}\quad
  \mathrm{H}(\nu \circ T^{-1}\mid \mu\circ T^{-1}) \leq \mathrm{H}(\nu\mid\mu).
\end{equation}
This contractivity argument is also at the heart of the lower bounds in
\cite{BCL23}. It follows that we can bound from below the relative entropy and
total variation mixing times by those of a suitable OU process, by taking a
multi-eigenfunction map. The lower bounds for OU processes have been
established for example in \cite{MR2203823}, \cite{MR2260742}, and
\cite[Th.~1.2]{BCL23}. Therefore, in the setting of
Corollary~\ref{co:cutoff:tv:H}, for all $c>0$, $\varepsilon\in(0,1)$,
\begin{equation}
  \lim_{d\to\infty}\sup_{x_0\in B(m, c\sqrt{d})} \DTV(X_{(1-\epsilon)t_*},\mu)=1
  \ \text{and}\ \lim_{d\to\infty}\sup_{x_0\in B(m, c\sqrt{d})} \mathrm{H}(X_{(1-\epsilon)t_*}\mid\mu)=\infty,
\end{equation}
when $X_0=x_0$, and for an arbitrary constant $c>0$. This is compatible with
\eqref{eq:CKP}. 

Another viewpoint for relative entropy is to get the Talagrand inequality 
\begin{equation}
  \mathrm{W}_2(X_t, \mu)^2 \leq \frac{2}{\rho}\mathrm{H}(X_t\mid\mu)
\end{equation}
from CD$(\rho, \infty)$, and deduce the relative entropy lower bound from the $\mathrm{W}_2$ lower bound. 

Let us prove the cutoff for the Fisher information.
Following \cite{BE85}, see also \cite{BGL14}, $\mathrm{CD}(\rho,\infty)$ gives the logarithmic Sobolev inequality  
\begin{equation}\label{eq:LSI}
  \mathrm{H}(\nu\mid\mu)\leq\frac{1}{2\rho}\mathrm{I}(\nu\mid\mu).
\end{equation}
Taking $\nu=\mathrm{Law}(X_t)$ and using the cutoff lower bound for $\mathrm{H}$ give, for all $\varepsilon\in(0,1)$, 
\begin{equation}\label{eq:I:lower}
  \lim_{d\to\infty}\sup_{x_0\in S}\mathrm{I}(X_{t_d}\mid\mu)=+\infty
  \quad\text{where}\quad
  t_d=(1-\varepsilon)t_*.
\end{equation}
For an upper bound, we can use another side of the Bakry--Émery theorem which states that $\mathrm{CD}(\rho,\infty)$ gives the monotonicities
\begin{equation}
  \partial_t\mathrm{H}(\mu_t\mid\mu)=-\mathrm{I}(\mu_t\mid\mu)\leq0
  \quad\text{and}\quad
  \partial_t\mathrm{I}(\mu_t\mid\mu)
  \leq-2\rho\mathrm{I}(\mu_t\mid\mu)\leq0
\end{equation}    
where $\mu_t$ and $\mu$ stand for $\mathrm{Law}(X_t)$ and
$\mu$. The Grönwall lemma implies then the exponential decay of
Fisher information: for all $t\geq t_1>0$,
\begin{equation}\label{eq:expdecI}
  \mathrm{I}(\mu_t\mid\mu)\leq\mathrm{e}^{-2\rho(t-t_1)}\mathrm{I}(\mu_{t_1}\mid\mu)
\end{equation}
(recall that $\mu_0$ is a Dirac mass). To upper bound $\mathrm{I}(\mu_{t_1}\mid\mu)$, we write, for all $0<t_0<t_1$,
\begin{equation}\label{eq:HI}
  \mathrm{H}(\mu_{t_0}\mid\mu)-\mathrm{H}(\mu_{t_1}\mid\mu)
  =\int_{t_0}^{t_1}\mathrm{I}(\mu_s\mid\mu)\mathrm{d}s
  \geq(t_1-t_0)\mathrm{I}(\mu_{t_1}\mid\mu)
\end{equation}
where we have used the monotonicity of $\mathrm{I}$. This gives, combined
with \eqref{eq:HtWX0}, when $t_0>1$,
\begin{equation} \label{eq_fisher_ent_w2}
  \mathrm{I}(\mu_{t_1}\mid\mu)
  \leq\frac{\mathrm{H}(\mu_{t_0}\mid\mu)}{t_1-t_0}
  \leq\frac{\mathrm{e}^{-2\rho(t_0-1)}}{2(t_1-t_0)}\mathrm{W}_2(X_0, \mu)^2.
\end{equation}
Finally, by using this inequality with $t_0=t_d-1$ and $t_1=t_d$ where
$t_d=(1+\varepsilon)t_*$ for an arbitrary fixed $\varepsilon>0$, using
\eqref{eq:WX0upper}, and $\rho=\lambda_1=o_{d\to\infty}(\log(d))$, we get
\begin{equation}\label{eq:I:upper}
  \lim_{d\to\infty}\sup_{x_0\in S}\mathrm{I}(X_{t_d}\mid\mu)=0.
\end{equation}

\begin{remark}[Weaker assumption]
  In the proof, the assumption $\lambda_1=o(\log(d))$ could have been replaced
  by $\lambda_1=d^{o(1)}$. This can be done by replacing, in the
  regularization step, the time $1$ by time $\asymp 1/\lambda_1$, which
  produces a prefactor $\lambda_1$ for $\mathrm{H}$, hence $\sqrt{\lambda_1}$
  for $\DTV$, and $\lambda_1^2$ for $\mathrm{I}$.
\end{remark}

\subsection{Proof of Theorem~\ref{thm_bg_convex_inter}}

Let $h(x)=x_1+\cdots+x_d$ be the first symmetric Hermite polynomial. Its
gradient is the constant vector $\nabla h=(1,\ldots,1)$. Since
$\mathrm{Hess}(V)=\rho I+\mathrm{Hess}(W)\geq\rho I$ as quadratic forms, we have
$\lambda_1 \geq \rho$. Hence, to get $\lambda_1 = \rho$, it suffices to show
that $h$ is an eigenfunction of $-\mathcal{L}$ with eigenvalue $\rho$, and
then apply the splitting theorem. The generator is
\begin{equation}
  \mathcal{L} f = \Delta f - \rho x \cdot \nabla f - \nabla W \cdot \nabla f.
\end{equation}
Since $W$ is invariant along $\mathbb{R}(1,\ldots,1)$, we have $\nabla W \cdot (1,\ldots,1) = 0$, therefore,
\begin{equation}
  \mathcal{L} h=0-\rho h-\nabla W\cdot(1,\ldots,1)=-\rho h.
\end{equation}
This gives $h\in E_1$ and $\lambda_1=\rho$. Note that by rigidity, all the elements of $E_1$ are affine.

\subsection{Proof of Corollary~\ref{cor_convex_inter}}

We drop the $(d)$ superscript. We have
$\sum_{i=1}^d m_i=\int\sum_{i=1}^d x_i\mathrm{d}\mu(x)=0$ since the image of
$\mu$ by $(x_1,\ldots,x_d)\mapsto x_1+\cdots+x_d$ is a centered Gaussian.

\emph{Lower bound when $t_d = (1-\epsilon)t_*$.} As per
Theorem~\ref{thm_bg_convex_inter}, $x_1 + \cdots + x_d$ is an eigenfunction.
Since its squared $L^2(\mu)$ norm is $d\rho^{-1}$,
$f_d(x) = \sqrt{\rho/d}(x_1 + \cdots + x_d)$ is a normalized eigenfunction.
Now, for both $S = B(m, c\sqrt{d})$ and $S = m + [-c,c]^d$ we have
\begin{equation}
\sup_{x \in S }|f_d(x)| = c\sqrt{\rho}\sqrt{d}.
\end{equation}
Hence the lower bound in Theorem~\ref{th:maineucl} yields 
\begin{equation}
  \sup_{x_0\in S}\mathrm{W}_2(X_{t_d}, \mu) \geq \mathrm{e}^{-\rho t_d}c\sqrt{d} = cd^{\epsilon/2} \xrightarrow[d\to\infty]{}+\infty.
\end{equation}

\emph{Upper bound when $t_d = (1+\epsilon)t_*$.} As in the proof of
Corollary~\ref{cor_cutoff_langevin_general}, the uniform convexity yields, via
the Poincaré inequality, for all $1\leq i\leq d$,
\begin{equation}
  \int(x_i - m_i)^2\mathrm{d}\mu(x) \leq \rho^{-1}.
\end{equation}
Hence when applying Theorem~\ref{th:maineucl} we get
\begin{equation}
  \sup_{x_0\in S}\mathrm{W}_2(X_{t_d},\mu) 
  \leq\mathrm{e}^{-\lambda_1t_d}\Bigl(\sup_{x_0\in S} |x_0-m|^2 + d/\rho\Bigr)^{1/2} 
  \leq (c^2 + \rho^{-1})^{1/2}d^{-\epsilon/2} 
  \xrightarrow[d\to\infty]{}0.
\end{equation}
It remains to proceed as for Corollary~\ref{cor_cutoff_langevin_general} and
Corollary~\ref{co:cutoff:tv:H} to conclude the proof.

\subsection{Proof of Theorem~\ref{th:prodcondcut}} 
Here again, we drop the $(d)$ superscript to simplify the notation. The
spectral gap $\lambda_1$ of $X$ satisfies $\lambda_1\geq\kappa$, thus
\begin{equation}
    \lim_{d\to\infty}\lambda_1t_0=+\infty,
\end{equation}
which is the product condition in \cite[Cor.~1]{salez-2025}, hence the cutoff in total variation distance. It remains to prove cutoff for the other cases. Let us start with the relative entropy lower bound. The Csiszár--Kullback--Pinsker inequality \eqref{eq:CKP} gives  
\begin{equation}
 \liminf_{d\to\infty}
 \sup_{x_0\in S}
 \mathrm{H}(X_{(1-\varepsilon/2)t_0}\mid\mu) \geq 2\eta^2.
\end{equation}
Moreover, from the exponential decay of entropy, for all $t>0$ and $t'\geq t$,
\begin{equation}
\mathrm{H}(X_{t'}\mid\mu)
\leq\mathrm{e}^{-2\kappa(t'-t)}\mathrm{H}(X_{t}\mid\mu).
\end{equation}
Using $t' = (1-\frac{\varepsilon}{2})t_0$, $t = (1-\varepsilon)t_0$, and
$\lim_{d\to\infty}\kappa t_0=+\infty$, we get 
\begin{equation}
  \liminf_{d\to\infty}\sup_{x_0\in S}\mathrm{H}(X_{(1-\varepsilon)t_0}\mid\mu)
  \geq\mathrm{e}^{\varepsilon\limsup_{d\to\infty}\kappa t_0}
  2\eta^2=+\infty.
\end{equation}
For the upper bound, from \cite[Proof of Th.~1]{salez-2025}, we get
\begin{equation}
\limsup_{d\to\infty}\sup_{x_0\in S}\mathrm{H}\bigl(X_{(1+\frac{\varepsilon}{2})t_0}\mid\mu\bigr) 
\leq C_{\eta,\varepsilon}<\infty.
\end{equation}
Using the exponential decay of the relative entropy and
$\lim_{d\to\infty}\kappa t_0=+\infty$, we get
\begin{equation}
    \limsup_{d\to\infty}\sup_{x_0\in S}\mathrm{H}(X_{(1+\varepsilon)t_0}\mid\mu)
    \leq\mathrm{e}^{-\varepsilon\liminf_{d\to\infty}\kappa t_0}
    C_{\eta,\varepsilon}
    =0.
\end{equation}

It remains to establish the cutoffs for $\mathrm{W}_2$ and $\mathrm{I}$. Let
us recall the extra assumptions $\varliminf_{d\to\infty}\kappa>0$
and $\lim_{d\to\infty}t_0=+\infty$.

Now, regarding $\mathrm{W}_2$, the upper bound comes from the one for
$\mathrm{H}$ via the Talagrand inequality
$\kappa\mathrm{W}_2(X,\mu)^2 \leq 2\mathrm{H}(X\mid\mu)$ and the assumption
$\varliminf_d\kappa>0$, while the lower bound comes from the regularization
\eqref{eq:HWsmallt} and the semigroup property that give, with
$s=\frac{\varepsilon}{2}t_0$,
\begin{equation}
  \mathrm{H}(X_{(1-\frac{\varepsilon}{2})t_0}\mid\mu)
  \leq\frac{1}{2s}\mathrm{W}_2(X_{(1-\varepsilon)t_0},\mu)^2,
\end{equation}
and then the lower bound for $\mathrm{H}$ and the assumption $t_0\to\infty$:
\begin{equation}
  \sup_{x_0\in S}\mathrm{W}_2(X_{(1-\varepsilon)t_0},\mu)^2
  \geq2s\times 2\eta^2=2\varepsilon\eta^2t_0\xrightarrow[d\to\infty]{}\infty.
\end{equation}
     
Finally, for $\mathrm{I}$, the lower bound comes from the one for $\mathrm{H}$
via the logarithmic Sobolev inequality
$2\kappa\mathrm{H}(X\mid\mu) \leq \mathrm{I}(X\mid \mu)$ and the assumption
$\varliminf_d\kappa>0$, while the upper bound comes by regularization from
\eqref{eq:HI} and the semigroup property to get
\begin{equation}
  \mathrm{I}(X_{t_1}\mid\mu)
  \leq\frac{\mathrm{H}(X_{t_0'}\mid\mu)}{t_1-t_0'}
\end{equation}
where $t_0'=(1+\frac{\varepsilon}{2})t_0$ and $t_1=(1+\varepsilon)t_0$, then
the upper bound on $\mathrm{H}$ and the assumption $t_0\to\infty$:
\begin{equation}
  \sup_{x_0\in S}\mathrm{I}(X_{(1+\varepsilon)t_0}\mid\mu)
  \leq\frac{\sup_{x_0\in
      S}\mathrm{H}(X_{(1+\frac{\varepsilon}{2})t_0}\mid\mu)}{\frac{\varepsilon}{2}t_0}
  \xrightarrow[d\to\infty]{}0.
\end{equation}

\subsection{Proof of Theorem~\ref{main_thm_riem}}
\label{se:riemann}

The proof is exactly the same as in the Euclidean setting
(Theorem~\ref{th:maineucl}), up to the use of the Riemannian splitting
(Theorem~\ref{thm_split_cz}) below, and the fact that first eigenfunctions are
affine through the splitting (Lemma~\ref{lem_rig}). We shall hence only
discuss these two elements, and omit the repetition of the proof. In this
setting, the rigidity theorem of \cite{CZ17} when $\lambda_1 = \rho$ is the
following.
 
\begin{theorem}[Riemannian splitting with Gaussian
  factor] \label{thm_split_cz} Let $(M, g, \mu)$ be a smooth, connected,
  complete weighted Riemannian manifold without boundary, with probability
  measure $\mu=\mathrm{e}^{-V}$. If for some
  $\rho>0$,
  \begin{equation}
  \text{$\mathrm{CD}(\rho, \infty)$ is satisfied and $\lambda_1=\rho$}, 
  \end{equation}
  then $(M,g, \mu)$ is isometric to a product weighted Riemannian manifold 
  \begin{equation}
    (\mathbb{R}^{k_1},\left|\cdot\right|_2,\gamma_{k_1,\rho^{-1}})\times (M',g',\mu')
    =(\mathbb{R}^{k_1} \times M',\left|\cdot\right|_2 \oplus g', \gamma_{k_1,\rho^{-1}} \otimes \mu')
  \end{equation}
  where 
  \begin{itemize}
  \item $k_1$ is the dimension of the eigenspace of $-\mathcal{L}$ associated with the eigenvalue $\lambda_1$
  \item $\gamma_{k_1,\rho^{-1}}$ is the centered Gaussian law on
    $\mathbb{R}^{k_1}$ with covariance matrix $\rho^{-1}\mathrm{Id}_{k_1}$
  \item $(M', g', \mu')$ is a smooth, connected, complete weighted Riemannian manifold without boundary satisfying $\mathrm{CD}(\rho, \infty)$,
    with $\mu'=\mathrm{e}^{-V'}\mathrm{vol}_{g'}$ for a smooth $V'$.
  \end{itemize}
  Moreover, if $\Phi:M\to\mathbb{R}^{k_1}\times M'$ denotes the isometry, then
  \[
    V\circ\Phi^{-1}(p,y)=\frac{\rho\left|p\right|_2^2}{2}+V'(y)+\frac{k_1}{2}\log\frac{2\pi}{\rho},
    \qquad p\in\mathbb{R}^{k_1},\ y\in M'.
  \]
\end{theorem}

This was proved by \cite{CZ17} in the setting
of complete smooth weighted manifolds, and by \cite{GKKO20} in the more general setting
of RCD spaces. Splitting theorems for manifolds satisfying a curvature
constraint and optimizing certain geometric quantities is a well-studied
problem in Riemannian geometry, going back to the Cheeger--Gromoll
splitting theorem for non-negatively curved manifolds containing infinite
geodesics.

We now state a structure lemma which enters the proof of Theorem~\ref{thm_split_cz}. 

\begin{lemma}[Rigidity for eigenfunctions when $\lambda_1 =
  \rho$] \label{lem_rig} Under the setting of Theorem~\ref{thm_split_cz}, and
  through the isometry that it provides, any element of $E_1$ is of the form
  $a \cdot p(x) $ with $p$ the projection on the Euclidean factor of dimension
  $k_1$, $a \in \mathbb{R}^{k_1}$, $|a| = \sqrt{\rho}\|f\|_{L^2(\mu)}$. Moreover, if
  $f_i = a_i \cdot p(x)$, for $i=1,2$, are two orthogonal eigenfunctions, then
  $a_1\perp a_2$.
\end{lemma}

\begin{proof}
  What follows is a broad sketch, where we focus on justifying the form of the
  eigenfunctions, but do not discuss in too much detail the splitting of the
  space, that was established in \cite{CZ17}. See also \cite{GKKO20} for a
  full proof in the non-smooth setting of RCD spaces. From the integrated
  Bochner formula, for any $g$ in the domain of $\mathcal{L}$ we have
  \begin{equation}
    \int-\Gamma(g,\mathcal{L} g)\mathrm{d}\mu 
    \geq \rho \int\Gamma(g)\mathrm{d}\mu + \int\|\mathrm{Hess}(g)\|_{\mathrm{HS}}^2\mathrm{d}\mu.
  \end{equation}
  Taking $g$ to be an eigenfunction with eigenvalue $\lambda_1 = \rho$, we get
  \begin{equation}
    \rho \int\Gamma(f)\mathrm{d}\mu
    \geq \rho \int\Gamma(f)\mathrm{d}\mu + \int\|\mathrm{Hess}(f)\|_{\mathrm{HS}}^2\mathrm{d}\mu
  \end{equation}
  which forces $\mathrm{Hess}(f) = 0$ almost everywhere, and thus everywhere
  since eigenfunctions are smooth. Therefore $f$ is affine, and it is
  non-constant whenever $f\ne0$.
  For a nonzero eigenfunction $f$, $\nabla f$ is a non-trivial parallel
  vector field of constant norm. Completeness ensures that its unit
  normalization has a flow defined for all real times. Along this flow,
  $f$ increases at the constant rate $|\nabla f|$, so every flow line meets
  $f^{-1}(0)$ exactly once. Since the field is parallel, its flow gives an
  isometry $\mathbb{R}\times f^{-1}(0)\to M$, along which $f$ is linear
  in the first coordinate. Repeating this
  strategy for successive orthogonal eigenfunctions gives a splitting with a
  $k_1$-dimensional Euclidean factor. See \cite{CZ17} for full details. And once
  the splitting is established, we can view the eigenfunctions as functions on
  $\mathbb{R}^{k_1}$.

  Since for an affine function $f = a\cdot x + b$, we have $\Gamma(f) = |a|^2$, and
  \begin{equation}
    |a|^2 =\int\Gamma(f)\mathrm{d}\mu = -\int f(\mathcal{L} f)\mathrm{d}\mu = \rho\|f\|_{L^2(\mu)}^2.
  \end{equation}
  Let us now compute $\Gamma(f_1, f_2)$ where $f_1, f_2$ are orthogonal
  eigenfunctions. We have
  \begin{equation}
    \Gamma(f_1, f_2) = \frac{1}{4}(\Gamma(f_1 + f_2) - \Gamma(f_1-f_2)) 
                     = \frac{\rho}{4}(\|f_1 + f_2\|_{L^2(\mu)}^2 - \|f_1 - f_2\|_{L^2(\mu)}^2) 
                     = 0
  \end{equation}
  where we used the fact that $f_1 \pm f_2$ are also eigenfunctions, with the
  same eigenvalue $\lambda_1$. But since
  $\Gamma(f_1, f_2) = \langle a_1, a_2\rangle$, the vectors driving $f_1$ and
  $f_2$ must be orthogonal.
\end{proof}

\end{document}